\documentclass[11pt]{amsart}

\usepackage{amssymb,fullpage}

\usepackage{color}
\usepackage[all]{xy}
\usepackage{enumitem}

\usepackage{graphicx}

\makeatletter
\@namedef{subjclassname@2020}{%
  \textup{2020} Mathematics Subject Classification}
\makeatother

\newtheorem{theorem}{Theorem}[section]
\newtheorem{corollary}[theorem]{Corollary}
\newtheorem{lemma}[theorem]{Lemma}
\newtheorem{proposition}[theorem]{Proposition}

\theoremstyle{definition}

\newtheorem{remark}[theorem]{Remark}
\newtheorem{example}[theorem]{Example}

\numberwithin{equation}{section}

\def\ld{\backslash}

\def\ker#1{\mathrm{ker}(#1)}
\def\aut{{\mathrm{Aut}}}
\def\lmlt{{\mathrm{LMlt}}}
\def\aff#1{\mathrm{Aff}#1}
\def\Adj{\mathrm{Adj}}

\def\aff#1{\mathrm{Aff}#1}

\def\Z{\mathbb Z}

\def\Q{\mathcal Q}
\def\sym{\mathrm{Sym}}
\def\dis{\mathrm{Dis}}

\def\setof#1#2{\{#1\, : \,#2\}}
\def\Sym{\mathrm{Sym}}
\def\c#1{\mathrm{con}_{#1}}

\def\m{\mathfrak{ip}}
\def\ldiv{\backslash}

\begin{document}

\title{A universal algebraic approach to rack coverings}

\author[Bonatto]{Marco Bonatto}

\address[Bonatto]{IMAS--CONICET and Universidad de Buenos Aires, Pabell\'on~1, Ciudad Universitaria, 1428, Buenos Aires, Argentina}

\email{marco.bonatto.87@gmail.com}

\author[Stanovsk\'{y}]{David Stanovsk\'{y}}

\address[Stanovsk\'{y}]{Department of Algebra, Faculty of Mathematics and Physics, Charles University, Prague, Czech Republic}

\email{stanovsk@karlin.mff.cuni.cz}

\date{}
\thanks{Research partly supported by the GA\v CR grant 18-20123S}


\date{\today}

\begin{abstract}
We study rack and quandle coverings from a universal algebraic viewpoint and we show how they can be understood using the notion of strongly abelian congruences. We provide an abstract characterization of several particular types of covering extensions, such as central and abelian ones.
We give a new characterization of simply connected quandles and we show that the categorical notion of normal extension coincides with the notion of central covering. We answer several questions from the papers of Clark, Saito and Vendramin \cite{CS} and \cite{CSV} about identities preserved by quandle coverings.
\end{abstract}

\maketitle

\section*{Introduction}\label{sec:intro}

{\it Racks and quandles} are algebraic structures related to knot invariants \cite{J}, set-theoretic solutions to the Yang-Baxter equation and Hopf algebras \cite{AG}. 
There have been several attempts to build a comprehensive theory of rack and quandle extensions, for instance, \cite{AG, CP, CENS}.
The present paper is a continuation of our project to study rack extensions using the universal algebraic framework of the commutator theory \cite{comm}. Here we focus on covering extensions.

\emph{Coverings}, closely related to \emph{extensions by constant cocycles}, are one of the most important types of rack extensions, mainly because constant cocycles provide powerful invariants of knots \cite{CJKLS, CSV, Eis-unknot}. Nevertheless, coverings are important also from the algebraic perspective. For example, every quandle is a cover of a conjugation quandle, using a Cayley-like representation. 
Even \cite{Even} studied coverings from a categorical viewpoint and proved that these are exactly the central extensions with respect to the adjunction between the category of quandles and the category of projection quandles.  
Eisermann \cite{Eisermann} developed a theory of quandle coverings in analogy with the covering theory of topological spaces using a categorical language. One of the many results of Eisermann is a characteriation of simply connected quandles, as connected quandles for which every covering is trivial. Our first result is an alternative group-theoretic characterization (Theorem \ref{caratt simply connected}).


We proceed with several universal algebraic aspects of rack extensions. The first one (Section~\ref{sec:identities}) is motivated by papers of Clark, Saito and Vendramin \cite{CS,CSV} who raised the question to what extent \emph{coverings preserve identities}. They focused on the subclass of abelian covering extensions, and on a special kind of identities called inner identities, written by composition of inner mappings (or left translations, in our terminology). We extend some of their results to arbitrary coverings and arbitrary identities, solving some of the open problems posted in \cite{CS,CSV}. Many identities are not preserved (Section \ref{negative examples}), however, every cover of a connected $n$-symmetric quandle is again $n$-symmetric (Theorem \ref{thm:symmetric}).

In Section \ref{sec:ua}, we show that coverings are captured by the universal algebraic notion of \emph{strongly abelian congruences}, and the derived notion of \emph{strongly solvable algebras}, as developed in \cite[Section 3]{TCT}. Covers are precisely the extensions over strongly abelian congruences (this is essentially contained in Proposition \ref{l:strongly_abelian iff under_lambda}).
Consequently, racks that are built from a trivial rack by a finite sequence of coverings (called \emph{multipermution racks} in the context of set-theoretic solutions of the Yang-Baxter equation) are exactly the strongly solvable racks, and we also prove that they are axiomatized by so called \emph{reductive laws} that were introduced in \cite{PR} (Theorem \ref{thm:strongly_solvable iff reductive}). 
Some of the results apply in a wider setting, for classes of left quasigroups where all terms admit a particular syntactic form, or where the Cayley kernel is always a congruence. This includes some of the other algebraic structures behind the Yang-Baxter equation, such as Rump's cycle sets \cite{BKSV,Rump}. 

In the rest of Section \ref{sec:ua}, we look at coverings which are also central in the sense of commutator theory \cite{CP,comm}.
First, we prove that every strongly solvable rack is nilpotent (Theorem \ref{nilpotency of lmlt and strongly solvability}). However, not every strongly abelian congruence is central, a characterization is given in Proposition \ref{central iff}. We also show that central coverings are the same thing as normal extensions with respect to the adjuction between the category of quandles and the category of projection quandles of \cite{Montoli,Even} (Theorem \ref{thm normal ext}).  

In section \ref{sec:ab cov}, we look at the concept of \emph{abelian extension} from \cite{CSV} and put it in our universal algebraic context. We find several abstract characterizations (Propositions \ref{abelian cov iff} and \ref{character ab cov}). 

The last section is dedicated to the properties of the extension over the largest idempotent factor.


The preliminaries on racks, quandles and extensions are summarized in the introductory Sections \ref{sec:prelim} and \ref{sec:coverings}. The proofs of all unproved statements can be found in \cite{GB,EN,hsv}. We refer to \cite[Section 4 and 5]{CP} for an introduction into commutator theory in the context of racks and quandles.

\section{Terminology and basic facts}\label{sec:prelim}

\subsection{Left quasigroups, racks and quandles}


A \emph{left quasigroup} is a binary algebraic structure $Q=(Q,*,\ld)$ such that  \[x\ast (x\ld y ) =y=x\ld (x\ast y)\] for every $x,y\in Q$.
The permutations $L_x:Q\to Q$, $y\mapsto x*y$, will be called \emph{(left) translations} (the adjective `left' will usually be dropped).

Two important permutation groups are associated to every left quasigroup:
the (left) \emph{multiplication group}, generated by all (left) translations,
\[\lmlt (Q)=\langle L_a:\ a\in Q\rangle\leq\mathrm{Sym}(Q),\]
and its subgroup, the \emph{displacement group}, defined by
\[\dis (Q)=\langle L_a L_b^{-1}:\ a, b\in Q\rangle\leq\lmlt(Q).\]
For $a\in Q$, we will denote $\dis(Q)_a$ the point stabilizer of $a$, and $\widehat{L_a}$ the automorphism of $\dis(Q)$ given by $\widehat{L_a}(\alpha)=L_a\alpha L_a^{-1}$.
We will denote $\sigma_Q$ the equivalence relation on $Q$ defined by
\[ a\,\sigma_Q\, b\ \Leftrightarrow\ \dis(Q)_a=\dis(Q)_b. \]
A left quasigroup $Q$ is called \emph{connected} if its left multiplication group acts transitively on $Q$. 

A left quasigroup is called a \emph{rack} if it is left distributive, i.e. if  \[ x*(y*z) =(x*y)*(x*z) \] holds for every $x,y,z\in Q$.
A \emph{quandle} is an idempotent rack, i.e. the identity $x*x=x$ holds for every $x\in Q$. In a rack, every translation is an automorphism.

\begin{example}
A rack $Q$ is called a \emph{permutation rack} if the operation does not depend on the first argument, i.e., the rack operation is
$a*b=f(b)$ where $f$ is a permutation of $Q$. If $Q$ is a quandle then $f$ is the identity on $Q$ and $Q$ is called a \emph{projection quandle}. If $|Q|=1$ then $Q$ is called {\it trivial}.
\end{example}

\begin{example}
Let $G$ be a group, $f$ its automorphism and $H$ a subgroup of $\mathrm{Fix}(f)$ (the subgroup of fixed points of $f$). We will denote $\mathcal{Q}(G,H,f)$ the quandle $(G/H,*,\ld)$ with the operations defined by 
$$aH\ast bH=af(a^{-1}b)H, \quad aH\ld bH=af^{-1}(a^{-1}b)H,$$ and call it a \emph{coset quandle}. If $H$ is the trivial group, $Q$ is called \emph{principal} over the group $G$ and denoted by $\mathcal{Q}(G,f)$. Moreover, if $G$ is abelian, then $Q$ is called \emph{affine} and an alternative notation, $\aff(G,f)$, is also used. 
\end{example}

A quandle is called \emph{principal} if it is isomorphic to $\mathcal Q(G,f)$, for some $G,f$. 
If $Q$ is a rack, then both $\dis(Q)$ and $\lmlt(Q)$ are normal subgroups of $\aut(Q)$. If $Q$ is a quandle the orbits of $\lmlt{(Q)}$ and of $\dis(Q)$ are the same, hence $Q$ is connected if and only if $\dis (Q)$ is transitive. 
Connected quandles can be represented as coset quandles over their displacement groups.
\begin{proposition}\cite[Proposition 3.5]{hsv}
Let $Q$ be a connected quandle and $a\in Q$. Then \[ Q\cong \mathcal{Q}(\dis(Q),\dis(Q)_a,\widehat{L_a}).\]
\end{proposition} 
Recall that a group acting on a set is {\it semiregular} if the pointwise stabilizers of the action are trivial. Transitive and semiregular groups are called {\it regular}.

\begin{proposition}\label{caratt principal}\cite[Proposition 2.1]{Principal} 
A connected quandle $Q$ is principal if and only if $\dis(Q)$ is regular on $Q$. In such a case, $Q\cong \mathcal{Q}(\dis(Q),\widehat{L_a})$.
\end{proposition}

\begin{remark}\label{connected by sub}
According to \cite[Lemma 3.11]{GiuThe}, the action of the displacement group of a coset quandle $Q=\mathcal{Q}(G,H,f)$ is given by the left action of $[G,f]=\langle xf(x)^{-1}: \ x \in G\rangle$ on the set $G/H$. If $Q$ is connected then $\dis(Q)=[\dis(Q),\widehat{L_a}]$ and $$Core_{\dis(Q)}(\dis(Q)_a)=\bigcap_{g\in \dis(Q)}\dis(Q)_{g(a)}=\bigcap_{b\in Q}\dis(Q)_b=1$$ since all the point stabilizers in $\dis(Q)$ are conjugate (recall that the core of a subgroup $H$ is the largest normal subgroup contained in $H$).
\end{remark}

\subsection{Congruences}

A congruence of a left quasigroup $Q$ is an equivalence relation on $Q$ invariant with respect to the operations $*,\ld$.
Let $f:Q\to R$ be a homomorphism. Its kernel, $\ker f=\setof{(x,y)\in Q^2}{f(x)=f(y)}$, is a congruence of $Q$. 
By virtue of the first isomorphism theorem, homomorphic images of $Q$ and congruences of $Q$ are essentially the same things. 

For a congruence $\alpha$ of $Q$, the blocks will be denoted by $[a]_\alpha$ (omitting the subscript whenever there is no risk of confusion) and the correspondent factor by $Q/\alpha$. If $Q/\alpha$ is connected then all blocks of $\alpha$ have the same cardinality and $\alpha$ is said to be {\it uniform}. Since the homomorphic image of a connected left quasigroup is connected, the congruences of connected left quasigroups are uniform. If $Q$ is idempotent, congruence blocks are subalgebras. If $Q$ is a connected quandle, the blocks of a given congruence are pairwise isomorphic \cite[Proposition 2.5]{CP}.

For every congruence $\alpha$ of a rack $Q$, there is a group homomorphism $\pi_\alpha:\lmlt (Q)\to\lmlt (Q/\alpha)$, defined by $L_a\mapsto L_{[a]_\alpha}$. 
Observe that, for every $h\in\lmlt (Q)$ and $x\in Q$, we have 
\begin{equation}\label{dag}
\pi_\alpha(h)([x]_\alpha)=[h(x)]_\alpha. 
\end{equation} 
Moreover, the mapping $\pi_\alpha$ restricts and corestricts to the displacements groups. We will denote the kernel of $\pi_\alpha$ by $\lmlt^\alpha$, and the kernel of its restriction by $\dis^\alpha$. Note that
\begin{eqnarray*}
\lmlt^\alpha&=&\setof{h\in \lmlt(Q)}{h(a)\,\alpha\ a \text{ for every }a\in Q}\label{kernel},\\
\dis^\alpha&=&\setof{h\in \dis(Q)}{h(a)\,\alpha\ a \text{ for every }a\in Q}\label{kernel of dis}.
\end{eqnarray*}
For $a\in Q$, we define the {\it block stabilizer} of $a$ by \[ \lmlt(Q)_{[a]_\alpha}=\pi_\alpha^{-1}(\lmlt(Q/\alpha)_{[a]_\alpha})=\setof{h\in \lmlt(Q)}{h(a)\,\alpha \, a}.\] 
The point stabilizer $\lmlt(Q)_a$ and the kernel $\lmlt^\alpha$ are contained in the block stabilizer. Similar inclusions holds for the subgroups $\dis(Q)_a$, $\dis^\alpha$ and $\dis(Q)_{[a]_\alpha}=\pi_\alpha^{-1}(\dis(Q/\alpha)_{[a]_\alpha})$. 

If $Q/\alpha$ is connected, then for every $a_0,a\in Q$ there exists $h\in \lmlt(Q)$ with $h([a_0])=h([a])$, and so
\begin{equation*}\label{kernel=core}
\lmlt^\alpha=\bigcap_{[a]_\alpha\in Q/\alpha} \lmlt(Q)_{[a]_\alpha}=\bigcap_{h\in \lmlt(Q)} h \lmlt(Q)_{[a_0]_\alpha}h^{-1}=Core_{\lmlt(Q)}(\lmlt(Q)_{[a_0]})
\end{equation*}
and similarly $\dis^\alpha=Core_{\lmlt(Q)}(\dis(Q)_{[a_0]})$. The following proposition gives a criterion for connectedness in terms of the action of the block stabilizers.

\begin{proposition}\label{prop:connected ext}\cite[Proposition 1.3]{GB} 
Let $Q$ be a rack (resp. a quandle) and $\alpha\in Con(Q)$. Then $Q$ is connected if and only if $Q/\alpha$ is connected and $\lmlt(Q)_{[a]_\alpha}$ (resp. $\dis(Q)_{[a]_\alpha}$) is transitive on $[a]_{\alpha}$ for every $a\in Q$.
\end{proposition}

In \cite{CP} we investigated the interplay between congruences and normal subgroups of the left multiplication group. For every congruence $\alpha$ of a rack $Q$ we define the \emph{displacement group relative to $\alpha$} as
\[\dis_{\alpha} =\langle L_a L_b^{-1}:\ a\, \alpha \, b\rangle.\]
If $\alpha=1_Q$ we recover the definition of the displacement group of $Q$. 
For every normal subgroup $N$ of $\lmlt(Q)$ we can define two congruences: let $\mathcal{O}_N$ be the orbit decomposition of $Q$ with respect to the action of $N$, i.e. $[a]_{\mathcal{O}_N}=\{g(a):g\in N
\}$ for every $a\in Q$, and let
\begin{eqnarray*}
\c{N}&=&\setof{(a,b)\in Q\times Q}{L_a L_b^{-1}\in N}.
\end{eqnarray*}
Clearly, $\mathcal{O}_N\leq \c{N}$.


\section{Covers, coverings and covering extensions}\label{sec:coverings}

\subsection{The Cayley kernel and covering homomorphisms}\label{ss:cayley}
The \emph{Cayley representation} of a left quasigroup $Q$ is the mapping \[ L_Q:Q\to\Sym(Q), \qquad x\mapsto L_x.\] 
Its kernel, \[\lambda_Q=\setof{(x,y)\in Q^2}{L_Q(x)=L_Q(y)}=\setof{(x,y)\in Q^2}{L_x=L_y},\] will be called the \emph{Cayley kernel} of $Q$. Left quasigroups with trivial Cayley kernel are called \emph{faithful}. If $Q/\alpha$ is faithful, then $\lambda_Q\leq \alpha$. Indeed, if $L_a=L_b$ then $L_{[a]_\alpha}=L_{[b]_\alpha}$ and so $[a]_\alpha=[b]_\alpha$ since $Q/\alpha$ is faithful.

The Cayley kernel is not always a congruence of $Q$. If it is, we call $Q$ a \emph{Cayley left quasigroup}. All racks are Cayley left quasigroups, and so are some of the other types of left quasigroups related to the set-theoretical solutions of the Yang-Baxter equation \cite{JPZ-retraction,Rump}. In this context, the Cayley kernel is known as the \emph{retraction relation}. The factor $Q/\lambda_Q$ is called the \emph{retract} of $Q$, and $Q$ is called \emph{$n$-multipermutational} if the $n$-th retract (i.e., the rectract of the rectract of \dots) is trivial \cite{ESS}.


Following \cite{Eisermann}, a \emph{covering homomorphism} is any surjective homomorphism of left quasigroups whose kernel is contained in the Cayley kernel. In other words, $f$ is a covering if $f(x)=f(y)$ implies $L_x=L_y$ for every $x,y$. A left quasigroup $Q$ is called a \emph{cover} of $R$ if there is a covering homomorphism $Q\to R$. In particular, $Q$ is a cover of $Q/\alpha$ whenever $\alpha\leq \lambda_Q$ (equivalently, if $\dis_\alpha=1$), using the natural projection $Q\to Q/\alpha$.

In racks, $L_Q$ is a homomorphism into $\aut(Q)$ with respect to the conjugation operation. The map $L_Q$ is the analog of the Cayley representation for groups but, unlike for groups, $L_Q$ is not necessarily injective. 

\begin{example}
Every rack is a cover of a conjugation quandle, using the Cayley homomorphism $L_Q$. 
\end{example}




For the Cayley kernel, the homomorphism $\pi_{\lambda_Q}$ extends to the automorphism group, using the expression \eqref{dag}.

\begin{proposition}\label{p:pi_aut}
Let $Q$ be a rack. Then the mapping
\[ \pi_{\lambda_Q}:	\aut(Q)\to \aut(Q/\lambda_Q), \qquad \pi_{\lambda_Q}(h)([x])=[h(x)] \]
is a well defined group homomorphism and $\ker{\pi_{\lambda_Q}}= C_{\aut(Q)}(\lmlt(Q))$.	
\end{proposition}

\begin{proof}
Let $h\in \aut(Q)$.	Note that the mapping $\pi_{\lambda_Q}(h)$ is well defined: if $a\, \lambda_Q \, b$, then $L_a=L_b$, and thus also $L_{h(a)}=h L_a h^{-1}=h L_b h^{-1}=L_{h(b)}$, hence $h(a)\, \lambda_Q\, h(b)$. Clearly, the mapping $\pi_{\lambda_Q}(h)$ is an automorphism of $Q/\lambda_Q$ and $\pi_{\lambda_Q}$ is a group homomorphism. To calculate the kernel, $\pi_{\lambda_Q}(h)=1$ if and only if $h(a)\,\lambda_Q\,a$ for every $a\in Q$, which means $L_{h(a)}=h L_a h^{-1}=L_a$ for every $a\in Q$, i.e. $h\in C_{\aut(Q)}(\lmlt(Q))$. 
\end{proof}




As a special case of Proposition \ref{p:pi_aut}, we obtain the following observation by Eisermann. 

\begin{corollary}\label{prop:covering_central}\cite[Proposition 2.49]{Eisermann}
Let $Q$ be a rack. Then $\lmlt^{\lambda_Q}=Z(\lmlt(Q))$.
\end{corollary}

According to Corollary \ref{prop:covering_central}, if $\alpha\leq \lambda_Q$ then $\lmlt^\alpha$ is contained in the center of $\lmlt(Q)$. The converse fails, even under the assumption that $Q$ is connected. For example, if $Q$ is the trivial quandle, then $E$ is a cover of $Q$ if and only if it is a projection quandle, while $\lmlt (E)$ is a central extension of $\lmlt (Q)$ if and only if it is an abelian group (for example, $E=\aff(\Z_4,-1)$ is not a projection quandle, but $\lmlt (E)$ is an abelian group).

\subsection{Rack cocycles and covering extensions}\label{cocycles}

Let $Q$ be a left quasigroup, $A$ a set, and $\theta:Q^2\to\Sym(A)$ a mapping into the symmetric group over $A$, to be called \emph{constant cocycle}. We will often denote the cocycle values $\theta(x,y)=\theta_{x,y}\in\Sym(A)$. Define a new operation on the set $E=Q\times A$ by 
\begin{align}\label{extensions by theta}
(x,a)*(y,b)&=(x*y,\theta_{x,y}(b)), \\
(x,a)\ldiv (y,b)&=(x*y,\theta_{x,x\ldiv y}^{-1}(b)).
\end{align}
The resulting left quasigroup $(E,\ast)$ is called the \emph{covering extension} of $Q$ over $\theta$, and denoted by $E=Q\times_\theta A$ (the name \emph{extension by constant cocycle} is also used in literature). 

The projection $\pi:Q\times_\theta A\to Q$, $(x,a)\mapsto x$ is a covering homomorphism and $\ker \pi$ is a uniform congruence. Conversely, if $f:E\to Q$ is a covering homomorphism and $\ker f$ is a uniform congruence (this is guaranteed whenever $Q$ is connected), then $E$ is isomorphic to a covering extension of $Q$ over $\theta$ (using the same argument as in \cite[Proposition 2.11]{AG}). In particular, let us identify $Q$ with $E/\ker f$, let $A$ be a set whose cardinality is the same as the cardinality of a block of $\ker f$, and let $h_{[x]}:[x]\to A$ be a family of bijections for $[x]\in Q=E/\ker f$. Then the mapping
\begin{equation}\label{standard cocycle}
\theta: Q \times Q \longrightarrow \Sym(A), \quad \theta_{[x],[y]} = h_{[x*y]} L_{x} h_{[y]}^{-1} 
\end{equation}
is well defined since $\ker f\leq \lambda_Q$ and the mapping
\begin{equation}\label{standard iso}
E\to Q\times_\theta A,\quad x\mapsto ([x] ,h_{[x]}(x))
\end{equation}
is an isomorphism of left quasigroups.

A covering extension $E$ is a rack if and only if $Q$ is a rack and $\theta$ satisfies the \emph{rack cocycle condition}
\begin{equation}\label{cocycle_condition}
\theta_{x,y\ast z}\circ\theta_{y,z}=\theta_{x\ast y,x\ast z}\circ\theta_{x,z},
\end{equation}
for every $x,y,z\in Q$ (here $\circ$ stands for composition of permutations). The rack $E$ is a quandle if and only if $Q$ is a quandle and $\theta_{x,x}=1$ for every $x\in Q$. 
We call the former $\theta$ a rack cocycle and the latter $\theta$ a quandle cocycle. In particular, \eqref{standard cocycle} defines a rack cocycle whenever $Q$ is a rack.  
%

The constant mapping $$\textbf{1}:Q\times Q\to \Sym(A),\quad (x,y)\mapsto 1$$ 
is always a rack cocycle, for every rack $Q$ and every set $A$. The rack $Q\times_{\textbf{1}} A$ is the direct product of $Q$ and the projection quandle over $A$ and it is called a {\it trivial covering extension} of $Q$.

Two cocycles $\theta$ and $\nu$ are called \emph{cohomologous} if there exists a mapping $\gamma:Q\to \sym{(A)}$ such that
\[\nu_{x,y}\circ \gamma_{y}=\gamma_{x\ast y}\circ \theta_{x,y}\]
for every $x,y\in Q$ (here $\circ$ stands for composition of permutations). The relation defined above is an equivalence on the set of all rack cocycles $Q^2\to\sym{(A)}$.
The set of its blocks is called the \emph{second rack cohomology set} of $Q$ over $A$ and denoted by $H^2(Q,A)$. See \cite[Section 2.1]{AG} for details. 

Cohomologous cocycles lead to isomorphic racks but the converse is not true. Nevertheless, in some cases the isomorphism problem is strictly related to the cohomology classes of cocycles.

\begin{lemma}\label{lemma on cohomology}
Let $Q$ be a rack, $\theta,\varepsilon$ rack cocycles over a set $A$, $E=Q\times_{\theta} A$, $E^\prime=Q\times_{\varepsilon} A$, $\pi$ be the canonical projection onto $Q$ and assume that $\lambda_{E}=\lambda_{E^\prime}=\ker{\pi}$.
Then $E\cong E^\prime$ if and only if there exists $g\in \aut(Q)$ such that $\theta$ and $\varepsilon\circ (g\times g)$ are cohomologous.
\end{lemma}

\begin{proof}
$(\Leftarrow)$ Let $\gamma:Q\to \sym{(A)}$ be a map such that ${\varepsilon}_{x,y}\circ \gamma_y=\gamma_{x*y} \circ {\theta}_{g(x),g(y)} $ for every $x,y\in Q$. Then the mapping $(x,a)\mapsto (g(x),\gamma_x(a))$ is an isomorphism.

	$(\Rightarrow)$ Assume that $f$ is an isomorphism between $E$ and $E^\prime$. Then the isomorphism $f$ induces an isomorphism $g:[a]_{\lambda_E}\mapsto [f(a)]_{\lambda_{E^\prime }}$ between $E/\lambda_E$ and $E^\prime/\lambda_{E^\prime}$. Indeed if $L_a=L_b$ then $L_{f(a)}=f L_{a}f^{-1}=f L_{b}f^{-1}=L_{f(b)}$. Therefore $f(x,a)=(g(x),\gamma_x(a))$ for every $x\in Q$ and $a\in A$, for certain $\gamma_x\in \sym{(A)}$. Then
	\begin{eqnarray*}
		f((x,a)*(y,b))&=&f(x*y,\theta_{x,y}(b))=(g(x*y),\gamma_{x*y}(\theta_{x,y}(b)))\\
		f(x,a)*f(y,b)&=&(g(x),\gamma_x(a))*(g(y),\gamma_y(b))=(g(x)*g(y),\varepsilon_{g(x),g(y)}(\gamma_y(b)))
	\end{eqnarray*}
	for every $x,y\in Q$ and $a,b\in A$. Therefore there exists a mapping $\gamma:Q\to\sym(A)$ such that ${\varepsilon}_{g(x),g(y)}\circ \gamma_y=\gamma_{x*y} \circ \theta_{x,y} $ for every $x,y\in Q$, i.e. $\theta$ and $\varepsilon \circ (g\times g)$ are cohomologous.
\end{proof}

\begin{corollary}
Let $Q$ be a faithful quandle and $\theta,\varepsilon$ rack cocycles. Then $Q\times_{\theta} A\cong Q\times_{\varepsilon} A$ if and only if there exists $g\in \aut(Q)$ such that $\theta$ and $\varepsilon\circ (g\times g)$ are cohomologous.
\end{corollary}

\begin{proof}
The factor $Q=E/\ker{\pi}$ is faithful and so $\lambda_E\leq \ker{\pi}$. On the other hand, $\ker{\pi}\leq \lambda_E$, since $E$ is an extension by a constant cocycle. Then $\lambda_{E}=\ker{\pi}$. The same is true for $E^\prime=Q\times_{\varepsilon} A$, therefore we can apply Lemma \ref{lemma on cohomology}.
\end{proof}

Consider the following particular type of cocycles. Let $A$ be endowed with an abelian group operation $+$, and assume that all permutations $\theta_{x,y}$ are translations of the group $(A,+)$. Then we can identify the permutations and the respective group elements, and redefine the cocycle as $\theta:Q^2\to A$, resulting in the operation
\[(x,a)\ast (y,b)=(x\ast y,b+\theta_{x,y})\]
on the set $E=Q\times A$. The cocycle condition then reads 
\[\theta_{x,y*z}+\theta_{y,z}=\theta_{x*y,x*z}+\theta_{x,z},\] 
and $\theta_{x,x}=0$ for every $x,y,z\in Q$ (addition in $A$ realizes composition of the corresponding translations).
This type of cocycles will be called \emph{abelian}, and we will talk about \emph{abelian covering extensions} (originally called \emph{abelian extensions} in \cite{CS,CSV}).

\section{Simply connected quandles}\label{sec:simply connected}

\subsection{Quandle coverings preserving the displacement group}\label{preserving dis}
%
A construction of quandle coverings based on the \emph{coset quandle} construction has been introduced in \cite[Proposition 2.13]{MeAndPetr} and it is a source of examples and counterexamples throughout the paper. 

\begin{lemma}\label{Prop:extension homogeneous}\cite[Propostion 2.13]{MeAndPetr} 
Let $G$ be a group, $H_{1}\leq H_{2}\leq \mathrm{Fix}(f)$, $Q_{H_i}=\mathcal{Q}(G,H_{i},f )$ and 
\begin{displaymath}
p: Q_{H_1}\to Q_{H_2},\quad aH_{1}\mapsto aH_{2}.
\end{displaymath}
Then $p$ is a covering homomorphism.
\end{lemma}
%
%
%
Let $Q$ be a rack and $\alpha$ be a congruence of $Q$. If $\pi_\alpha$ is an isomorphism then $\dis_\alpha\leq \dis^{\alpha}=1$ and so $Q$ is a cover of $Q/\alpha$ and they have isomorphic displacement groups. For connected quandles we can prove that such covers have a particular form.

\begin{proposition}\label{Prop: Extensions with same Trans} 
Let $Q$ be a connected quandle and $\alpha$ be a congruence of $Q$. The following conditions are equivalent:
\begin{enumerate}
\item[(i)] $\pi_\alpha$ is a group isomorphism.
\item[(ii)] $Q\simeq (Q/\alpha)_H=\mathcal{Q} ( \dis ( Q/\alpha ) ,H,\widehat{L_{[a]}})$
for some $H\leq \dis(Q/\alpha)_{[a]}$.
\end{enumerate}
\end{proposition}

\begin{proof}
(i) $\Rightarrow$ (ii) Let $
h\in \dis(Q)_a$, then $\pi_\alpha ( h) ( [a])
=[ h(a)] =[a]$, hence $H=\pi_\alpha(\dis(Q)_a) \leq \dis(Q/\alpha)_{[a]} $. Therefore the mapping
\begin{displaymath}
 \mathcal{Q} ( \dis ( Q) ,\dis ( Q)_{a},\widehat{L_a}) \longrightarrow \mathcal{Q} (
\dis ( Q/\alpha ) ,H, \widehat{L_{[a]}}), \quad h \dis(Q)_a \mapsto \pi_\alpha(h) H
\end{displaymath}
is a well defined isomorphism of quandles.

(ii) $\Rightarrow$ (i) Let $Q_H=\Q( \dis ( Q/\alpha ) ,H,\widehat{L_{[a]}})$
for some $H\leq \dis(Q/\alpha)_{[a]}$. According to Remark \ref{connected by sub}, $Q_H$ is connected since the action of $\dis(Q_H)$ is the canonical left action of $\dis(Q/\alpha)=[\dis(Q/\alpha),\widehat{L}_{[a]}]$ on the set of cosets $\dis(Q/\alpha)/H$.
By Lemma \ref{Prop:extension homogeneous}, the mapping
\[p:Q_H\mapsto\mathcal{Q}(\dis(Q/\alpha),\dis(Q/\alpha)_{[a]},\widehat{L_{[a]}}),\quad bH\mapsto b\dis(Q/\alpha)_{[a]}\]
is a surjective quandle homomorphism and $Q_H$ is a cover of $Q/\alpha$. Moreover $\pi_\alpha$ is an isomorphism if and only if $\pi_{\ker{p}}$ is an isomorphism. The action of $\pi_{\ker{p}}(h)$ is given by the left action of some $t_h\in \dis(Q/\alpha)$, i.e.
\[ \pi_{{\ker{p}}}(h)(b \dis(Q/\alpha)_{[a]})=t_h b \dis(Q/\alpha)_{[a]}\] for every $b\in \dis(Q/\alpha)$. So $h\in \dis^{\ker p}$ if and only if $t_h\in Core_{\dis(Q/\alpha)}(\dis(Q/\alpha)_{[a]})=1$ (see Remark \ref{connected by sub} again). Therefore $h=1$ and $\pi_{\ker{p}}$ is a group isomorphism.
%
\end{proof}
%
If $Q$ is finite, then $\dis(Q)$ and $\dis(Q/\alpha)$ are finite groups and so in Proposition \ref{Prop: Extensions with same Trans}, we can replace condition (i) by the condition $\dis(Q)\simeq \dis(Q/\alpha)$. Indeed, since $\pi_\alpha$ is a surjective morphism, if $\dis(Q)\simeq \dis(Q/\alpha)$ then $\pi_\alpha$ is also injective.
%
%
%
%
%
%
%
%

\subsection{Characterization of simply connected quandles}\label{ss:simply connected}
A natural problem about quandle coverings is to characterize quandles for which every cover is trivial (i.e. isomorphic to the direct product with a projection quandle). In \cite{Eisermann} this problem has been tackled using a categorial approach with a particular focus on the \emph{adjoint group} of a quandle \cite[Definition 2.18]{Eisermann}, defined by
$$\Adj(Q)=\langle e_x:\,x\in Q  \ |\ e_x e_y e_x^{-1}=e_{x*y}:\, x,y\in Q\rangle.$$ 
This group is also called {\it enveloping group} in \cite{QuadraticNichols} and {\it structure group} in the framework of the solutions of the Yang-Baxter equation \cite{ESS}.

According to \cite{Eisermann}, there exists a group homomorphism $\varepsilon: \Adj(Q)\to \mathbb{Z}$ mapping every generator to $1$ and $\Adj(Q)\cong \Adj(Q)^0\rtimes \mathbb{Z}$ where $\Adj(Q)^0=\ker{\varepsilon}$. The map $L_Q$ factors through $\Adj(Q)$: indeed there exists a surjective group homomorphism $\overline{L_Q}: \Adj(Q)\to \lmlt(Q)$ such that the following diagram is commutative
\begin{equation}
\xymatrixcolsep{63pt}\xymatrixrowsep{30pt}\xymatrix{ Q\ar[dr]^{L_Q}
\ar[r]^{\iota} & \Adj(Q)\ar[d]^{\overline{L_Q}} \\ & \lmlt(Q) }  \label{Diag:Adjoint}
\end{equation}
where $\iota$ maps every element of $Q$ to the correspondent generator of $\Adj(Q)$. In this way we obtain an action of $\Adj(Q)$ on $Q$ as $g\cdot a=\overline{L_Q}(g)(a)$ for every $g\in\Adj(Q)$ and $a\in Q$. In particular $\overline{L_Q}(\Adj(Q)^0)=\dis(Q)$ and $g\in \Adj(Q)_a^0$ if and only if $\overline{L_Q}(g)\in \dis(Q)_a$. With abuse of notation we denote with the same symbol the homomorphism $\overline{L_Q}$ and its restriction to $\Adj(Q)^0$, whose image is the displacement group of $Q$.

A quandle $Q$ is called {\it simply connected} if it is connected and $\Adj(Q)^0_a=1$ for every $a\in Q$ \cite[Definition 5.14]{Eisermann}. Eisermann proved in \cite[Proposition 5.15]{Eisermann} that simply connected quandles are precisely the connected quandles for which all cocycles are cohomologous, i.e., $|H^2(Q,A)|=1$ for every set $A$.
We show an alternative characterization in terms of the relation between $\Adj(Q)$ and $\dis(Q)$. 

\begin{theorem}\label{caratt simply connected}
Let $Q$ be a connected quandle $Q$. The following are conditions equivalent: 
\begin{enumerate}
\item[(i)] $Q$ is simply connected.
\item[(ii)] $|H^2(Q,A)|=1$ for every set $A$.
\item[(iii)] $Q$ is principal and $\overline{L_Q}$ is an isomorphism. 
\end{enumerate}
\end{theorem}
\begin{proof}
The equivalence of (i) and (ii) is proved in \cite[Proposition 5.15]{Eisermann}.

The inclusion $ \ker{\overline{L_Q}}\leq \Adj(Q)^0_a$ holds for any quandle $Q$ and every $a\in Q$. According to Proposition \ref{caratt principal}, $Q$ is principal if and only if $\dis(Q)_a=1$ for every $a\in Q$. Therefore $Q$ is principal if and only if  $\Adj(Q)^0_a\leq \ker{\overline{L_Q}}$, i.e. $\Adj(Q)^0_a= \ker{\overline{L_Q}}$ for every $a\in Q$.
	
(i) $\Rightarrow$ (iii) Assume that $Q$ is simply connected. By Proposition \ref{Prop: Extensions with same Trans}, $E=\mathcal{Q}(\dis(Q),\widehat{L_a})$ is a connected cover of $Q$. Hence, $E\cong Q\times P$ where $P$ is a projection quandle. Therefore $|P|=|\dis(Q)_a|=1$ i.e. $Q$ is principal. Then $\ker{\overline{L_Q}}=\Adj(Q)^0_a=1$, and so $\overline{L_Q}$ is an isomorphism. 

(iii) $\Rightarrow$ (i) If $Q$ is a principal connected quandle and $\overline{L_Q}$ is an isomorphism, then $\ker{\overline{L_Q}}=\Adj(Q)^0_a =1$ and so $Q$ is simply connected.
\end{proof}
The conditions in Theorem \ref{caratt simply connected}(iii) are independent as witnessed by some examples in the RIG library \cite{RIG}: indeed for $Q=$ {\tt SmallQuandle}(6,1) the mapping $\overline{L_Q}$ is an isomorphism but $Q$ is not principal and  {\tt SmallQuandle}(25,$i$), $1\leq i \leq 5$ are affine, but $\overline{L_Q}$ is not an isomorphism.

\begin{corollary}\label{adjoint for simply connected}
Let $Q$ be a simply connected quandle then $\Adj(Q)\cong \dis(Q)\rtimes \mathbb{Z}$.
\end{corollary}
%
The converse of Corollary \ref{adjoint for simply connected} holds for finite connected quandles. Indeed if $Q$ is finite and connected then $\Adj(Q)^0$ is also finite \cite[Lemma 2.19]{QuadraticNichols}. The mapping $\overline{L_Q}$ is surjective and then also injective since $|\dis(Q)|=|\Adj(Q)^0|$. There are infinite counterexamples, see \cite[Example 1.25]{Eisermann}.




Connected quandles with a cyclic displacement group and connected quandles with doubly transitive displacement group are simply connected \cite{MeAndPetr} and simply connected quandles of size $p^2$ have been classified in \cite{Vendramin_p_square_cohomology}. Table \ref{Tab2} collects all the other simply connected quandles up to size $47$ which do not fall in these families (the data have been computed using Theorem \ref{caratt simply connected} and the RIG library \cite{RIG}). Simply connected quandles need not to be faithful, as witnessed by some of the quandles in Table \ref{Tab2}.

\begin{center}
\begin{table}
\caption{Simply connected quandles up to size $47$.}
\label{Tab2}
\begin{tabular}{|c|l|}
\hline
  Size & {\tt SmallQuandle}(Size,-)   \\
 \hline
8 &  1\\
 24 & 1, 2, 8, 24, 25 \\
   27 & 1, 6, 14, 27, 28, 29, 30, 31, 32, 33, 34\\
    40 & 4, 5, 6, 27, 28, 29, 30, 31, 32\\ 
    45& 38, 39, 40, 41, 42, 43  \\
 \hline
\end{tabular}
\end{table}
\end{center}

\section{Coverings and identities}\label{sec:identities}

\subsection{Terms and identities}

A \emph{term} $t=t(x_1,\dots,x_n)$ is a well-formed formal expression using the variables $x_1,\dots,x_n$ and the left quasigroup operations $\{*,\ld\}$.
We will often omit parentheses, assuming implicitly the right parenthesizing; for example, $x*y*z\,\ld\, u*v$ will stand for $x*(y*(z\,\ld\,(u*v)))$. Occasionally, we will use juxtaposition for terms involving just $*$, e.g. $zxyz$ will stand for $z*(x*(y*z))$.

Formally, an \emph{identity} is a pair of terms, to be written as $t=s$. Two terms $t,s$ are called equivalent in a structure $A$ (in a class $\mathcal C$, resp.) if the identity $t=s$ holds in $A$ (in every structure in $\mathcal C$, resp.).




Following \cite{CS}, an identity is called \emph{inner} if it has the form 
\[z_1\bullet_1 z_2\bullet_2 \dots \bullet_{m-1}z_m\bullet_m y=y,\] 
where $z_1,\dots,z_m$ are selected arbitrarily from a set of variables $x_1,\dots,x_n$, excluding $y$, and all $\bullet_i\in\{*,\ld\}$. 
The \emph{symmetric laws} are a particular example: a quandle is called \emph{$n$-symmetric} if it satisfies the identity 
\begin{equation*}
\underbrace{x*x*\ldots*x}_{n-\text{times}}*\,y=y.
\end{equation*} 
Another example is \emph{mediality}. It is usually defined as the identity $(x*y)*(u*v)=(x*u)*(y*v)$, but for racks it is easily proved to be equivalent to abelianness of the displacement group \cite[Proposition 2.4]{hsv}, which can be written as the inner identity $x*y\,\ld\, u*v\,\ld\, y*x\,\ld\, v*u\,\ld\,z=z$. 

Following \cite{PR}, a left quasigroup $Q$ is called \emph{$n$-reductive}, if the composition of any $n$ right translations is a constant mapping, i.e., if the expression 
$((\ldots((u*x_1)*x_2)\ldots)*x_{n-1})*x_n$
does not depend on the choice of $u$. Equivalently, if $Q$ satisfies the \emph{reductive law}
\[ ((\ldots((u*x_1)*x_2)\ldots)*x_{n-1})*x_n = ((\ldots((v*x_1)*x_2)\ldots)*x_{n-1})*x_n. \]
A left quasigroup is 1-reductive if and only if it is permutational. It is easy to check that 2-reductive racks are medial, but there exist non-medial 3-reductive quandles.
Reductive medial quandles were studied extensively in \cite[Sections 6, 8]{JPSZ}.
(In recent quandle literature, unfortunately, reductivity is defined in several different ways that are equivalent for medial quandles, but not in general. We decided to refer back to the original source \cite{PR} which explains the phenomenon of reductivity in the class of algebraic structures called modes.) 

\begin{lemma}\label{l:reductive} 
Let $Q$ be a Cayley left quasigroup. Then $Q$ is $n$-reductive if and only if $Q/\lambda_Q$ is $(n-1)$-reductive.
\end{lemma}

\begin{proof}
The reductive law is equivalent to stating that \[ (\ldots((u*x_1)*x_2)\ldots)*x_{n-1} \ \lambda_Q\ (\ldots((v*x_1)*x_2)\ldots)*x_{n-1}\] for every $x_1,\dots,x_{n-1}\in Q$. 
\end{proof}

\begin{corollary}\label{simple reductive}
There is no non-trivial connected $n$-reductive quandle. 
\end{corollary}

\begin{proof}
If $n=1$, then $Q$ is a connected projection quandle, hence trivial. If $n>1$ then $Q/\lambda_Q$ is connected and $(n-1)$-reductive, hence $\lambda_Q=1_Q$ by the induction assumption, and again, $Q$ is a connected projection quandle, hence trivial. 
\end{proof}


\subsection{When covering preserves an identity?}

Let $t$ be a term in variables $x_1,\dots,x_n,y$ where $y$ is the rightmost. We define a formal expression $\Theta_t$ in variables $x_1,\dots,x_n,y$ using new symbols $\theta,\circ,^{-1}$ recursively, by 
\[ \Theta_y=1, \qquad \Theta_{t*s}=\theta(t,s)\circ\Theta_s,\qquad \Theta_{t\ldiv s}=\theta(t,t\ldiv s)^{-1}\circ\Theta_s. \]
In particular, if $t=t_1t_2\dots t_my$, $m\geq 1$, is a term that only uses $*$, then
\[\Theta_t=\theta(t_1,t_2\dots t_m y)\circ\theta(t_2,t_3\dots t_m y)\circ\dots\circ\theta(t_{m-1},t_m y)\circ\theta(t_m,y).\]
For a particular left quasigroup $Q$, cocycle $\theta$, and a choice of $x_1,\dots,x_n,y$ from $Q$, we will treat the expression $\Theta_t$ as composition of the respective values of $\theta$ or its inverses.
(This construction extends to left quasigroups \cite[Definition 4.2]{CS} and it generalizes it to more general terms and to non-abelian cocycles.)

\begin{example}
For $t=\underbrace{x\dots x}_{n}y$, we have $t_1=\ldots=t_n=x$ and 
\[\Theta_t=\theta(x,\underbrace{x\dots x}_{n-1}y)\circ\theta(x,\underbrace{x\dots x}_{n-2}y)\circ\ldots\circ\theta(x,xy)\circ\theta(x,y).\]
\end{example}

\begin{example}
For $t=(xy)(zw)$, we have $t_1=xy$, $t_2=z$ and $\Theta_t=\theta(xy,zw)\circ\theta(z,w)$.
\end{example}

\begin{lemma}\label{l:Theta}
Let $t$ be a term in variables $x_1,\dots,x_n,y$ where $y$ is the rightmost.
Let $Q$ be a left quasigroup and consider the covering extension $E=Q\times_\theta A$.
Then for every $(u_1,a_1),\dots,(u_n,a_n),(v,b)\in E$,
\[t((u_1,a_1),\dots,(u_n,a_n),(v,b)) = (t(u_1,\dots,u_n,v),\Theta_t(u_1,\dots,u_n,v)(b)).\]
\end{lemma}

\begin{proof}
Straightforward induction. For $m=0$, we have $t=y$, $\Theta_t=1$ and $t(v,b)=(v,b)$. For $t=s*r$, we have
\begin{align*} 
t((u_1,a_1),\dots,(u_n,a_n),(v,b)) &= s((u_1,a_1),\dots,(u_n,a_n),(v,b))*r((u_1,a_1),\dots,(u_n,a_n),(v,b))  \\
& = (s(u_1,\dots,u_n)*r(u_1,\dots,u_n),\ \theta_{s(u_1,\dots,u_n),r(u_1,\dots,u_n)}(\Theta_r(u_1,\dots,u_n,v)(b)) \\
&= (t(u_1,\dots,u_n),\ \Theta_t(u_1,\dots,u_n,v)(b))
\end{align*}
using the induction assumption for $r$ in the second step, and the recursive definition of $\Theta_t$ in the last step.
For $t=s\ldiv r$ we proceed similarly.
\end{proof}

The following proposition describes when a covering extension satisfies an identity $t=s$ where both terms $t,s$ have the same rightmost variable (such as mediality or $n$-symmetry). It generalizes \cite[Theorem 4.2(ii)]{CS}, which addresed the special case of abelian covering extensions and identities without left division.

\begin{proposition}\label{p:sat}
Let $t=s$ be an identity in variables $x_1,\dots,x_n,y$ where both terms $t,s$ have the same rightmost variable $y$.
Let $Q$ be a left quasigroup, and consider the covering extension $E=Q\times_\theta A$.
The following statements are equivalent:
\begin{enumerate}
	\item[(i)] $E$ satisfies the identity $t=s$;
	\item[(ii)] $Q$ satisfies the identity $t=s$ and the equality $\Theta_t(u_1,\dots,u_n,v)=\Theta_s(u_1,\dots,u_n,v)$ holds for every $u_1,\dots,u_n,v\in Q$.
\end{enumerate}
\end{proposition}

\begin{proof}
The proof follows immediately from Lemma \ref{l:Theta}.
\end{proof}

\begin{example}
Assume that $Q$ is $n$-symmetric. A covering extension of $Q$ over $\theta$ is $n$-symmetric if and only if 
\[\Theta_{x\dots xy}(u,v)=\theta(u,\underbrace{u\dots u}_{n-1}v)\circ\theta(u,\underbrace{u\dots u}_{n-2}v)\circ\ldots\circ\theta(u,uv)\circ\theta(u,v)=1\] 
for every $u,v\in Q$.
\end{example}

\begin{example}
Assume that $Q$ is medial.  A covering extension of $Q$ over $\theta$ is medial if and only if 
\[\Theta_{(xy)(zw)}(r,s,t,u)=\theta(rs,tu)\circ\theta(t,u)=\theta(rt,su)\circ\theta(s,u)=\Theta_{(xz)(yw)}(r,s,t,u)\] 
for every $r,s,t,u\in Q$.
\end{example}

\begin{remark}
A non-trivial idempotent cover cannot satisfy any identity $t=s$ where the rightmost variables are different (such as commutativity): the blocks of the Cayley kernel are projection quandles, and they fail such identities. Therefore, we exclude such identities from our study.
\end{remark}

\subsection{Negative examples}\label{negative examples}

We are not aware of any interesting identity preserved by rack coverings in general. For example, 2-symmetry is not preserved, the order of translations in the cover can exceed any finite bound. The following example answers \cite[Question 8.8]{CSV} negatively.

\begin{example}
Let $Q$ be the quandle defined by the following multiplication table:
\[\begin{array}{|c|ccc|}\hline
 &0&1&2\\\hline
0&0&2&1\\
1&0&1&2\\
2&0&1&2\\\hline
\end{array}\]
Clearly, $Q$ is 2-symmetric.
Let $A$ be an abelian group, fix $a\in A$, and define $\theta:Q^2\to A$ by $\theta(0,2)=a$ and $\theta(x,y)=0$ for all pairs $(x,y)\neq(0,2)$. 
It is straightforward to verify that $\theta$ is an abelian cocycle, for any parameter $a \in A$.
Consider the covering extension $E=Q\times_\theta A$. The order of the translation $L_{(0,0)}$ in $\lmlt(E)$ is $n=2\mathrm{ord}(a)$, twice the order of $a$ in $A$, since 
\[(1,0)\mapsto(2,0)\mapsto(1,a)\mapsto(2,a)\mapsto(1,2a)\mapsto(2,2a)\mapsto\ldots\mapsto(1,(n-1)a)\mapsto(2,(n-1)a)\mapsto(1,0).\]
Therefore, the cover is not $m$-symmetric for any $m<n$. (It is easy to check that the cover is actually $n$-symmetric.)
\end{example}

Reductivity is not satisfied either (clearly, non-trivial covers of 1-reductive racks are 2-reductive), however, we have the following semi-positive result.

\begin{proposition}
Let $Q$ be a rack satisfying the identity $s=t$. Then every cover of $Q$ satisfies the identity $s*z=t*z$
where $z$ is a new variable. In particular, every cover of an $n$-reductive rack is $(n+1)$-reductive.
\end{proposition}

\begin{proof}
It follows directly from Proposition \ref{p:sat}, since $\Theta_{t*z}=\theta(t,z)=\theta(s,z)=\Theta_{s*z}$.
\end{proof}

From now on, we will focus on connected racks. Not surprisingly, many identities are not preserved. We will show three examples: mediality, a short identity in two variables, and a short inner identity.

\begin{example}
Let $Q=\aff(\Z_2^2,f)$ be the (unique) connected quandle of order 4. It is medial and it satisfies the identities $xyyxy=y$ and $xyxyxyz=z$ (the latter is the "ababab" identity from \cite{CS}). 
But $Q$ has a connected abelian covering extension over the group $\Z_2$, namely the (unique) non-affine connected quandle of order 8, which fails each of the three identities, and also fails $xyxyxy=y$.
(The information can be collected from various calculations in \cite{CS,CSV}, or one can verify the claim directly using the explicit construction of the 8-element quandle in \cite[Example 8.6]{hsv}.)
 

\end{example}

\subsection{Positive examples}\label{sec:symmetric law}

The following result shows that certain type of identities is preserved in connected covers. It includes the symmetric laws, thus solving \cite[Conjecture 5.2]{CS}.

\begin{proposition}\label{inner identity preserved}
Let $Q$ be a rack satisfying the identity $t=y$ where \[ t(x_1,\dots,x_n,y)=x_{i_1}\bullet_1\dots\bullet_{m-1} x_{i_m}\bullet_m y.\] 
Let $E$ be a connected rack cover of $Q$ such that for every $a_1,\ldots,a_n\in E$ there exists $b\in E$ such that $t(a_1,\dots,a_n,b)=b$.  
Then $E$ satisfies the identity $t=y$.
\end{proposition}

\begin{proof}
Let $\alpha$ denote the kernel of the projection $E\to Q$. Fix $a_1,\dots,a_n\in E$ and consider the mapping $h=L_{a_1}^{k_1}\ldots L_{a_n}^{k_n}\in\lmlt(E)$ where $k_i=1$ if $\bullet_i=*$, and $k_i=-1$ otherwise. Indeed, $t(a_1,\dots,a_n,e)=h(e)$ for every $e\in E$. We shall prove that $h$ is the identity mapping.

By assumption, $h$ has a fixed point.
Since $Q\simeq E/\alpha$ satisfies the identity $t=y$, we have $h\in\lmlt^\alpha\leq\lmlt^{\lambda_E}=Z(\lmlt(E))$, using Corollary \ref{prop:covering_central} in the last equality. As every central subgroup of a transitive group is semiregular we have that $h=1$.
%
\end{proof}

In quandles, the fixed point assumption holds for any identity in just two variables, since $t(a,a)=a$ for every term $t$ and element $a$. In particular, it applies to the symmetric laws.

\begin{corollary}\label{corollary for connected}
Every connected quandle cover of an $n$-symmetric quandle is $n$-symmetric.
\end{corollary}

\begin{proof}
Apply the previous proposition to $t(x,y)=x\dots xy$.
\end{proof}

The corollary can be generalized to arbitrary covers of connected quandles.

\begin{theorem}\label{thm:symmetric}
Every quandle cover of a connected $n$-symmetric quandle is $n$-symmetric.
\end{theorem}

\begin{proof}
Let $\alpha$ denote the kernel of the projection $E\to Q$. Fix $a,b\in E$. Since $Q$ is $n$-symmetric, we have $L_a^n\in\lmlt^\alpha\leq\lmlt^{\lambda_E}=Z(\lmlt(E))$, using Corollary \ref{prop:covering_central} in the last equality. 
Since $E/\alpha$ is connected, there is $h\in\lmlt(E)$ which maps $[a]_\alpha$ into $[b]_\alpha$. In particular, $h(a)\,\alpha\,b$, and since $\alpha\leq\lambda_E$, we have $L_{h(a)}=L_b$. 
Since $L_a^n$ is central, $L_b^n=L_{h(a)}^n=hL_a^n h^{-1}=L_a^n$ and thus $L_b^n(a)=L_a^n(a)=a$.
\end{proof}

\begin{corollary}
Let $Q$ be a connected quandle. Then $Q$ is $n$-symmetric if and only if $\widehat{L_a}$ has order $n$ in $\aut{(\dis(Q))}$.
\end{corollary}

\begin{proof}
According to Lemma \ref{Prop:extension homogeneous} the quandle $E=\mathcal{Q}(\dis(Q),\widehat{L_a})$ is a connected cover of $Q$. The quandle $E$ is $n$-symmetric if and only if the order of $\widehat{L_a}$ is $n$. 
Obviously, if $E$ is $n$-symmetric, then so is every homomorphic image. Conversely, if $Q$ is $n$-symmetric, then so is $E$ by Theorem \ref{thm:symmetric}.
\end{proof}

\begin{remark}
Under sufficiently strong assumptions, many identities are preserved. The extremal case is that of simply connected quandles (this includes, for example, connected affine quandles over cyclic groups \cite{MeAndPetr}). Indeed, trivial covers preserve any identity with the same rightmost variable. 
\end{remark}

\section{A universal algebraic characterization of coverings}\label{sec:ua}

\subsection{Strong abelianness and strong solvability}\label{sec:strongly abelian}
Let $\alpha\geq\beta$ be congruences of an algebraic structure $A$. Following \cite[Section 3]{TCT}, we say that $\alpha$ is \emph{strongly abelian} over $\beta$, if for every term $t(x,y_1,\dots,y_n)$, every pair $u\,\alpha\,v$ and all $a_i,b_i,c_i\in A$ such that $a_i\,\alpha\,b_i\,\alpha\,c_i$, for every $i=1,\dots,n$,
\[ t(u,a_1,\dots,a_n)\ \beta\ t(v,b_1,\dots,b_n)\quad\text{implies}\quad t(u,c_1,\dots,c_n)\ \beta\ t(v,c_1\dots,c_n). \]
A congruence is called \emph{strongly abelian}, if it is strongly abelian over the smallest congruence, $0_A$.  
An algebraic structure $A$ is called \emph{strongly abelian} if its largest congruence, $1_A$, is strongly abelian.

An algebraic structure $A$ is called \emph{strongly solvable of length at most $n$}, if it possesses congruences 
$0_A=\alpha_0\leq\alpha_1\leq\ldots\leq\alpha_n=1_A$ such that $\alpha_{i+1}$ is strongly abelian over $\alpha_i$, for every $i$.
Observe that if $\xi\leq\beta\leq\alpha$ are congruences of $A$ then $\alpha$ is strongly abelian over $\beta$ in $A$ if and only if $\alpha/\xi$ is strongly abelian over $\beta/\xi$ in $A/\xi$
(in particular, if and only if $\alpha/\beta$ is strongly abelian in $A/\beta$). Therefore, strong solvability carries over to factors.

Permutation racks are strongly abelian, since every term depends only on one variable. Conversely, every strongly abelian rack is a permutation rack: considering the term $t(x,y)=x\ast y$ and any $u,v$, we have $t(u,u\ld v)=t(v,v\ld v)$, and strong abelianness gives $t(u,w)=t(v,w)$ for every $w$, which means $L_u=L_v$. 
This observation can be generalized.

\begin{proposition}\label{l:strongly_abelian iff under_lambda} 
Let $Q$ be a left quasigroup and $\alpha$ a congruence of $Q$. The following conditions are equivalent: 
\begin{itemize}
\item[(i)] $\alpha$ is strongly abelian;
\item[(ii)] $\alpha\leq \lambda_Q$.
\end{itemize}
\end{proposition}

\begin{proof}
(i) $\Rightarrow$ (ii).
Assume that $a\,\alpha\,b$. Then also $(a\ast c)\,\alpha\,(b\ast c)$ and $c\, \alpha\, (a\ld(b\ast c))$ for every $c\in Q$. Applying strong abeliannes to the term $t(x,y)=x\ast y$ and equality $t(a,a\ld(b*c))=t(b,c)$, we obtain $t(a,c)=t(b,c)$ for every $c$, i.e., $L_a=L_b$ and $a\,\lambda_Q\, b$.



%
%

(ii) $\Rightarrow$ (i). 
We verify that $\alpha$ is strongly abelian.
Let $t(x,y_1,\dots,y_n)$ be a term and write it in the form \[ t=s_1\bullet_1(s_2\bullet_2(\ldots(s_{m-1}\bullet_{m-1}s_m)))\] 
where $s_1,\dots,s_m$ are terms, $s_m$ is a single variable, and all $\bullet_j\in \{*,\ldiv\}$. 
Let $u\,\alpha\,v$ and $a_i\,\alpha\,b_i\,\alpha\,c_i$ for every $i=1,\dots,n$. Observe that for every $j=1,\dots,m$
\[ s_j(u,a_1,\dots,a_n)\,\alpha\,s_j(v,b_1,\dots,b_n) \text{ and } s_j(u,c_1,\dots,c_n)\,\alpha\,s_j(v,c_1,\dots,c_n).\]
Since $\alpha\leq\lambda_Q$, the respective left translations coincide.

Now assume that $t(u,a_1,\dots,a_n)=t(v,b_1,\dots,b_n)$. Since $L_{s_j(u,a_1,\dots,a_n)}=L_{s_j(v,b_1,\dots,b_n)}$ for every $j=1,\dots,m-1$, it follows that $s_m(u,a_1,\dots,a_n)=s_m(v,b_1,\dots,b_n)$.
But $s_m$ is a single variable, hence $s_m(u,c_1,\dots,c_n)=s_m(v,c_1,\dots,c_n)$, too. Since $L_{s_j(u,c_1,\dots,c_n)}=L_{s_j(v,c_1,\dots,c_n)}$ for every $j=1,\dots,m-1$, we obtain $t(u,c_1,\dots,c_n)=t(v,c_1,\dots,c_n)$.
\end{proof}



\begin{theorem}\label{thm:strongly_solvable iff reductive}
Let $Q$ be a left quasigroup such that all homomorphic images of $Q$ are Cayley left quasigroups. Then the following conditions are equivalent:
\begin{enumerate}
	\item[(i)] $Q$ is strongly solvable of length at most $n$,
	\item[(ii)] $Q$ is $n$-reductive,
	\item[(iii)] $Q$ is $n$-multipermutational.
\end{enumerate}
\end{theorem}
\begin{proof}
We proceed by induction on $n$. For $n=1$, the theorem says that $Q$ is strongly abelian if and only if it is 1-reductive if and only if $\lambda_Q=1_Q$. The latter equivalence is obvious and the former one is a special case of Proposition \ref{l:strongly_abelian iff under_lambda} for $\alpha=1_Q$. Now assume that the theorem holds for all $k<n$.

$(i)\Rightarrow(ii)$
Assume that $Q$ is strongly solvable of length at most $n$, witnessed by the chain $0_Q=\alpha_0\leq\alpha_1\leq\ldots\leq\alpha_n=1_Q$.
Then $Q/\alpha_1$ is strongly solvable of length at most $n-1$, witnessed by the chain $0_{Q/\alpha_1}=\alpha_1/\alpha_1\leq\ldots\leq\alpha_n/\alpha_1=1_{Q/\alpha_1}$. By the induction assumption, $Q/\alpha_1$ is $(n-1)$-reductive, hence
\[(\ldots((u\ast x_1)*x_2)\ldots)*x_{n-1} \, \alpha_1 \, (\ldots((v*x_1)*x_2)\ldots)*x_{n-1}\]
for all $u,v,x_1,\dots,x_{n-1}\in Q$. Since $\alpha_1$ is strongly abelian, we have $\alpha_1\leq\lambda_Q$ by Proposition \ref{l:strongly_abelian iff under_lambda}, hence $Q/\lambda_Q$ is $(n-1)$-reductive, and thus $Q$ is $n$-reductive by Lemma \ref{l:reductive}.

$(ii)\Rightarrow(iii)$
If $Q$ is $n$-reductive, then $Q/\lambda_Q$ is $(n-1)$-reductive by Lemma \ref{l:reductive}, hence $(n-1)$-multipermutational by the induction assumption, and thus $Q$ is $n$-multipermutational. 

$(iii)\Rightarrow(i)$
Assume that $Q$ is $n$-multipermutational. Then $Q/\lambda_Q$ is $(n-1)$-multipermutational, hence, by the induction assumption, it is strongly solvable of length at most $n-1$, witnessed by a chain $0_{Q/\lambda_Q}=\lambda_Q/\lambda_Q\leq\alpha_1/\lambda_Q\leq\ldots\leq\alpha_n/\lambda_Q=1_{Q/\lambda_Q}$. Then $0_Q\leq\lambda_Q\leq\alpha_1\leq\ldots\leq\alpha_n=1_Q$ witnesses that $Q$ is strongly solvable of length at most $n$.
\end{proof}

\begin{corollary}
A connected rack $Q$ is strongly solvable if and only if it can be constructed from the trivial rack by repeated covering extensions.
\end{corollary}


\subsection{Centrality and nilpotence}\label{sec:central cov}

 
In \cite{CP}, we demonstrated that commutator properties (in the sense of \cite{comm}) of rack congruences correspond nicely to commutator properties of the corresponding relative displacement groups. 
We refer to \cite{CP,comm} for the abstract definitions of central congruences and nilpotent algebraic structures. In the present paper, we will use the rack-theoretic characterization of the two notions, given by the next two propositions. Recall that $a\,\sigma_Q\, b$ iff $\dis(Q)_a=\dis(Q)_b$.

\begin{proposition}\label{central congruences} \cite[Theorem 1.1 and Proposition 5.9]{CP}
	Let $Q$ be a rack and $\alpha$ be a congruence of $Q$. Then $\alpha$ is a central congruence if and only if $\dis_{\alpha}$ is contained in the center of $\dis(Q)$ and $\alpha\leq \sigma_Q$. 
\end{proposition}

A covering is called \emph{central} if its kernel is a central congruence.

\begin{proposition}\label{nilpotency} \cite[Lemma 6.2]{CP} 
	A rack $Q$ is nilpotent if and only if the group $\dis(Q)$ is nilpotent. If $\dis(Q)$ is nilpotent of length $n$, then $Q$ is nilpotent of length at most $n+1$.
\end{proposition}

We start with an observation regarding nilpotence of racks constructed by repeating covering extensions.
In general, strongly solvable algebraic structures are not necessarily nilpotent. In racks, the situation is different. 

\begin{theorem}\label{nilpotency of lmlt and strongly solvability}
Every rack which is strongly solvable of length $n$ is also nilpotent of length at most $n$. 
\end{theorem}

\begin{proof}
Let $Q$ be a strongly solvable rack of length $n$. If $n=1$ then $Q$ is a permutation rack, which is nilpotent of length 1. Let $n\geq 2$. Then $Q$ is $n$-multipermutational by Theorem \ref{thm:strongly_solvable iff reductive}, and according to \cite[Main Theorem]{MultRackSol}, the group $\lmlt(Q)$ is nilpotent of length at most $n-1$. Then $\dis(Q)$ is also nilpotent of length at most $n-1$, and thus $Q$ is nilpotent of length at most $n$ by Proposition \ref{nilpotency}.
\end{proof}
%

The bound on the length of nilpotence given in Theorem \ref{nilpotency of lmlt and strongly solvability} is not optimal: every medial $n$-reductive quandle is nilpotent of length $2$ \cite[Proposition 5.13]{CP}. 

For congruences of racks, the situation is different: strongly abelian congruences are not necessarily central. To characterize rack congruences which are central and strongly abelian at the same time, look at Proposition \ref{central congruences} under the assumption that $\alpha$ is strongly abelian, i.e., $\dis_\alpha=1$: then $\alpha$ is central if and only if $\alpha\leq \sigma_Q$. For connected quandles, we have a nicer characterization.

\begin{proposition}\label{central iff}
Let $Q$ be a connected quandle and $\alpha\leq\lambda_Q$. The following conditions are equivalent:
\begin{itemize}
\item[(i)] $\alpha$ is central.
\item[(ii)] $\alpha\leq \sigma_Q$.
\item[(iii)] $\dis(Q)_a\unlhd \dis(Q)_{[a]_\alpha}$ for every $a\in Q$.
\end{itemize}
\end{proposition}

\begin{proof}
%
According to Proposition \ref{prop:connected ext} the block $[a]_\alpha$ coincides with the orbit of $a$ with respect to the action of the block-stabilizer. Therefore for every $b\, \alpha\,a $ there exists $h\in \dis(Q)_{[a]}$ such that $b=h(a)$.

(ii) $\Rightarrow$ (iii) If $h\in \dis(Q)_{[a]}$ then $h(a)\, \alpha\, a$ and so $\dis(Q)_{h(a)}=h \dis(Q)_a h^{-1}=\dis(Q)_a$ and so $h\in N_{\dis(Q)}(\dis(Q)_a)$.

(iii) $\Rightarrow$ (ii) Since $\dis(Q)_a\unlhd \dis(Q)_{[a]}$ then for every $b\, \alpha\, a$ we have $\dis(Q)_b=\dis(Q)_{h(a)}=h \dis(Q)_a h^{-1}=\dis(Q)_a$, hence $\alpha\leq \sigma_Q$.
\end{proof}

In certain classes of racks, strongly abelian congruences are central. For example, whenever $\dis(Q)$ is semiregular, which happens, for example, for every principal quandle.

\begin{proposition}\label{strongly abelian -> central}
Let $Q$ be a rack such that $\dis(Q)$ is semiregular, and let $\alpha\leq\lambda_Q$. Then $\alpha$ and $\c{\dis^{\alpha}}$ are central congruences. 
\end{proposition}

\begin{proof}
Since $\dis(Q)$ is semiregular, we have $\sigma_Q=1_Q$, and thus a congruence $\beta$ is central if and only if $\dis_\beta\leq Z(\dis(Q))$.
Since $\alpha\leq \c{\dis^\alpha}$, we have
\begin{equation*}
\dis_{\c{\dis^\alpha}}\leq \dis^\alpha\leq \dis^{\lambda_Q}= Z(\lmlt(Q))\cap \dis(Q)\leq Z(\dis(Q))
\end{equation*}
and so $\alpha$ and $\c{\dis^\alpha}$ are central.
\end{proof}
%

The following example shows that factors of central congruences are not necessarily central.

\begin{example}
Let $Q$ be a connected quandle, $a\in Q$, $1=H_0\leq H_1\leq H_2\leq \dis(Q)_a$ and $Q_i=\Q(\dis(Q),H_i,\widehat{L_a})$ for $i=0,1,2$. By Lemma \ref{Prop:extension homogeneous}, $Q_i$ is connected, $\dis(Q)=\dis(Q_i)$ and moreover if $i\leq j$ then 
	$$p_{i,j}:Q_i\longrightarrow Q_{j},\quad gH_i\mapsto gH_j$$ 
is a covering homomorphism. The stabilizer $\dis(Q_0)_1$ is trivial since $Q_0$ is principal and so it is normal in the block stabilizer of $\ker{p_{0,i}}$ for $i=1,2$. According to Proposition \ref{central iff}, $Q_0$ is a central cover of both $Q_{1}$ and $Q_{2}$, i.e. $\ker{p_{0,i}}$ is a central congruence for $i=1,2$. Let $Q_{2}\cong Q_{1}/\ker{p_{1,2}}$. The stabilizer in $\dis(Q_1)$ of the element $H_1\in Q_1$ is  the subgroup $H_1$ and the block stabilizer of $[H_1]_{\ker{p_{1,2}}}$ is $H_2$. Using Proposition \ref{central iff} again, we obtain that $Q_{1}$ is a central cover of $Q_{2}$ if and only if $H_1\unlhd H_2$. Thus, if $H_1$ is not normal in $H_2$, the congruence $\ker{p_{1,2}}=\ker{p_{0,2}}/\ker{p_{0,1}}$ is a factor congruence of a central congruence, but it is not central itself.
\end{example}

%
%
%

%
%
%
%
%
%

\subsection{Central and strongly abelian is normal.}

In \cite{Montoli, Even}, Even et al. investigated the categorical concepts of \emph{central extensions} and \emph{normal extensions} with respect to the adjunction between the category of quandles and the category of projection quandles (their notion of central extension is different from the one that comes from commutator theory \cite{CP}). According to \cite[Theorem 2]{Even}, central extensions of quandles in the sense of \cite{Montoli, Even} are the same as quandle coverings. We will prove that normal extensions in the sense of \cite{Montoli, Even} are the same as central coverings in the sense of the previous subsection. We will use the following characterization of normal extensions.

\begin{lemma}\cite[Proposition 3.2]{Montoli}
Let $Q$ be a quandle. Then $Q$ is a normal extension of $Q/\alpha$ if and only if for all $a_i\, \alpha\, b_i$
\begin{equation}\label{normal ext}
L_{a_1}^{k_1}\ldots L_{a_n}^{k_n}(a_{n+1})=a_{n+1} \, \Longrightarrow \, L_{b_1}^{k_1}\ldots L_{b_n}^{k_n}(b_{n+1})=b_{n+1}.
\end{equation}
\end{lemma}

\begin{theorem}\label{thm normal ext}
Let $Q$ be a quandle. Then $Q$ is a normal extension of $Q/\alpha$ if and only if $\alpha$ is strongly abelian and central.
\end{theorem}

\begin{proof}
$(\Rightarrow)$ Every normal extension is central \cite{Montoli} and so $L_a=L_b$ whenever $a\,\alpha\,b$. So we can write the property \eqref{normal ext} as
	\begin{equation}\label{semiregularity}
	h=L_{a_1}^{k_1}\ldots L_{a_n}^{k_n}\in \lmlt(Q)_a \, \Rightarrow \, h|_{[a]_\alpha}=1.
	\end{equation}
	which is exactly $\alpha$-semiregularity of $\lmlt(Q)$. Therefore $\dis(Q)$ is $\alpha$-semiregular too, and so $\alpha$ is a central congruence.
	
	$(\Leftarrow)$ Assume that $\dis(Q)$ is $\alpha$-semiregular and let $h=gL_a^k\in \lmlt(Q)_a$ for some $k\in\mathbb{Z}$ and $g\in \dis(Q)_a$. If $a\,\alpha\, b$ we have that $h(b)=gL_a^k(b)=g(b)=b$, since $\dis(Q)$ is $\alpha$-semiregular. Therefore \eqref{semiregularity} holds and so does \eqref{normal ext}, i.e. $Q$ is a normal extension of $Q/\alpha$.
\end{proof}

\section{Abelian covering extensions}\label{sec:ab cov}

%

\emph{Abelian covering extensions} are special cases of two general constructions: \emph{covering extensions} as defined in Section \ref{cocycles}, and {\it central extensions} as defined in \cite[Section 7]{CP}. 
It follows from  \cite[Proposition 7.5]{CP} that, in an abelian covering extension $Q\times_\theta A$, the kernel of the canonical projection onto $Q$ is a central congruence.

We say that a rack $E$ is an \emph{abelian cover} of a rack $Q$ if $E$ is isomorphic to an abelian covering extension of $Q$. In this section, we characterize abelian covers. Universal algebra does not seem to provide a good concept, but there is a convenient characterization in terms of the structure of the automorphism group.

\begin{proposition}\label{abelian cov iff} 
Let $Q$ be a rack and $\alpha\leq \lambda_Q$. The following conditions are equivalent:
\begin{itemize}
\item[(i)] $Q$ is an abelian cover of $Q/\alpha$.
\item[(ii)] There exists an abelian subgroup $A\leq \aut(Q)$ such that $A_a=1$ and $[a]_\alpha=a^A$ for every $a\in Q$.
\end{itemize}
\end{proposition}
\begin{proof}
(i) $\Rightarrow$ (ii). 
Assume that $Q$ equals $Q/\alpha\times_{\theta} \tilde A$ for some abelian group $\tilde A$ and some cocycle $\theta$. 
For $a\in\tilde A$, we define a permutation $\rho_a$ of $Q$ by $\rho_a(x,s)=(x,s+a)$ and let $A=\{\rho_a:a\in\tilde A\}$. We see that $A$ is an abelian subgroup of $\aut(Q)$, since $\rho_a\rho_b=\rho_{a+b}$ and 
\begin{eqnarray*}
\rho_a((x,s)\ast (y,t))&=&\rho_a(x*y, t+\theta_{x,y})=(x*y, t+\theta_{x,y}+a)\\&=&(x,s+a)*(y,t+a)=\rho_a(x,s)*\rho_a(y,t)
\end{eqnarray*} 
for every $a,b\in \tilde A$ and $(x,s),(y,t)\in Q$.
Clearly we have $A_{(x,s)}=1$ and $[(x,s)]_\alpha=(x,s)^A$ for every $(x,s)\in Q$.

%

(ii) $\Rightarrow$ (i).  
We will find an abelian cocycle $\theta$ such that $Q$ is isomorphic to $Q/\alpha\times_\theta A$. 
Let $\setof{e_B}{B\in A/\alpha}$ be a set of representatives of the blocks of $\alpha$. The group $A$ is regular on each $\alpha$-block $B$ and so the map 
$$h_{B}:A\longrightarrow B,\quad s\mapsto s(e_B)$$
is a bijection. Let $\theta_{[a],[b]}$ be the unique element of $A$ such that $$\theta_{[a],[b]}(e_{[a*b]})=e_{[a]}*e_{[b]}.$$ 
We will show that the mapping $a\mapsto ([a], h_{[a]}^{-1}(a))$ is an isomorphism. We can apply the idea of Section \ref{cocycles}, cf. equations \eqref{standard cocycle} and \eqref{standard iso}, as long as we verify that $s\theta_{[a],[b]}=h_{[a*b]}^{-1}L_{e_{[a]}}h_{[b]}(s)$ for every $s\in A$. Indeed, using $s L_{e_{[a]}}=L_{e_{[a]}}s$ for every $s\in A$, we get 
\begin{eqnarray*}
h_{[a*b]}^{-1}L_{e_{[a]}}h_{[b]}(s) &=& h_{[a*b]}^{-1} (e_{[a]}*s(e_{[b]})) \\ &=& h_{[a*b]}^{-1} s(e_{[a]}*e_{[b]}) = h_{[a*b]}^{-1}(s\theta_{[a],[b]}(e_{[a*b]})) = s\theta_{[a],[b]}.\qquad\qquad\qedhere
\end{eqnarray*}
\end{proof}

%

Proposition \ref{abelian cov iff} shows that connected abelian covers are exactly the {\it Galois covers} as defined in \cite[Definition 4.12]{Eisermann} for which the group of deck transformations is abelian.
 
As a consequence of Proposition \ref{abelian cov iff}, we prove the following adaptation of \cite[Proposition 3.1]{EG} to the context of coverings. (For quandles, the statement follows from \cite[Proposition 7.8]{CP}.)

\begin{corollary}\label{p:kerf*_trans}
Let $Q$ be a connected rack and $\alpha\leq \lambda_Q$ a congruence. Then $Q$ is an abelian cover of $Q/\mathcal{O}_{\lmlt^\alpha}$ and $\lmlt(Q/\mathcal{O}_{\lmlt^\alpha})\simeq\lmlt(Q/\alpha)$.
\end{corollary}

The statement of the corollary can be represented by the following diagrams
\begin{equation*}
\xymatrixcolsep{63pt}\xymatrixrowsep{30pt}\xymatrix{ Q\ar[dr]^{}
\ar[r]^{} & Q/\mathcal{O}_{\lmlt^\alpha} \ar[d]^{} \quad \\ & Q/\alpha }  \qquad \xymatrixcolsep{63pt}\xymatrixrowsep{30pt}\xymatrix{ \lmlt(Q)\ar[dr]^{\pi_\alpha}
\ar[r]^{\pi_{\mathcal{O}_{\lmlt^\alpha}}} & \lmlt(Q/\mathcal{O}_{\lmlt^\alpha})\ar[d]^{\phi} \\ & \lmlt(Q/\alpha) }  
\end{equation*}
where $\phi$ is an isomorphism and $Q$ is an abelian cover of $Q/\mathcal{O}_{\lmlt^\alpha}$.

\begin{proof}
Denote $H=\lmlt^\alpha$ and $\beta=\mathcal{O}_H$.
The subgroup $H$ is contained in $Z(\lmlt(Q))\cap \aut^\alpha$. Since $H$ is contained in the center of a transitive group, it is abelian and semiregular. Its orbits coincide with the blocks of $\beta$, i.e. $a^H=[a]_\beta$ and $H_a=1$ for every $a\in Q$. Therefore we can apply Proposition \ref{abelian cov iff} and so $Q$ is an abelian cover of $Q/\beta$.  

Since $\beta\leq \alpha$ then $\lmlt^{\beta}\leq\lmlt^{\alpha}$. If $h\in \lmlt^\alpha$ then clearly $h(a)\, \beta\, a$ for every $a\in Q$, therefore $\lmlt^\alpha=\lmlt^\beta$. Thus, $\lmlt(Q/\alpha)\cong \lmlt(Q/\beta)$.
\end{proof}

For connected racks, we can characterize abelian covers in terms of properties of the congruence block stabilizers of the group $\lmlt(Q)$.

\begin{proposition}\label{character ab cov}
	Let $Q$ be a connected rack and $\alpha\leq \lambda_Q$. The following conditions are equivalent: 
	\begin{itemize}
		\item[(i)] $Q$ is an abelian cover of $Q/\alpha$.
		\item[(ii)] $(\lmlt(Q)_{[a]})|_{[a]}$ is abelian for every $a\in Q$.
	\end{itemize}
\end{proposition}

From the proof, one can notice that the implication (i) $\Rightarrow$ (ii) holds for arbitrary quandles, not necessarily connected.

\begin{proof}
	(i) $\Rightarrow$ (ii) Assume that $Q$ equals $Q/\alpha\times_\theta A$ for some abelian group $A$ and some cocycle $\theta$. 
Consider $h\in \lmlt(Q)_{[a]}$. Then $ h(a,s)=(a,s+\Theta_{h,a})$, where $\Theta_{h,a}$ does not depend on $s$. Hence
	$$gh(a,s)=(a,s+\Theta_{h,a}+\Theta_{g,a} )=hg(a,s), $$
	for every $h,g\in \lmlt(Q)_{[a]}$. Therefore $(\lmlt{(Q)}_{[a]})|_{[a]}$ is abelian for every $a\in Q$.
	%
%
	
	(ii) $\Rightarrow$ (i) Let $\setof{e_B}{B\in Q/\alpha}$ be a set of representatives of the blocks of $\alpha$. Let $e\in Q$ be a fixed element in $Q$. Let $\setof{g_{[b]}}{[b]\in Q/\alpha}\subseteq \lmlt(Q)$ be a set of mappings such that $g_{[b]}:[e]\longrightarrow [b]$ (such mappings exist since $Q$ is connected). The group $A=(\lmlt(Q)_{[e]})|_{[e]}$ acts regularly on $[e]$ (transitivity follows from Proposition \ref{prop:connected ext} and semiregularity from abelianness). Hence the mapping
	$$h_{[b]}:A\to [b], \quad s\mapsto g_{[b]}s(e)$$
is a bijection. Since $g^{-1}_{[a*b]}L_{e_{[a]}}g_{[b]}\in \lmlt(Q)_{[e]}$, define $\theta_{[a],[b]}=(g^{-1}_{[a*b]}L_{e_{[a]}}g_{[b]})|_{[e]}$. 
As in the proof of Proposition \ref{abelian cov iff}, the mapping $a\mapsto ([a]_\alpha,h_{[a]}^{-1}(a)) $ is an isomorphism $Q\to Q\times_{\theta} A$, as long as we verify, for every $s\in A$, 
\begin{eqnarray*}
h_{[a*b]}^{-1} L_{e_{[a]}} h_{[b]}(s) &=& h_{[a*b]}^{-1} L_{e_{[a]}} g_{[b]}s(e) = h_{[a*b]}^{-1} g_{[a*b]} g_{[a*b]}^{-1} L_{e_{[a]}} g_{[b]}s(e)\\
&=& h_{[a*b]}^{-1} g_{[a*b]} s g_{[a*b]}^{-1} L_{e_{[a]}} g_{[b]}(e) = h_{[a*b]}^{-1} g_{[a*b]} s \theta_{[a],[b]}(e)= s\theta_{[a],[b]}.\quad\qedhere
\end{eqnarray*}
\end{proof}

Using Proposition \ref{character ab cov}, we can provide examples of central coverings which cannot be realized by abelian covering extensions.

\begin{example}\label{non-abelian-but-central} 
Let $Q=\mathcal{Q}(\dis(Q),\dis(Q)_a,\widehat{L_a})$ be a connected quandle. Let $E=\mathcal{Q}(\dis(Q),\widehat{L_a})$ and $p:E\to Q$ the quandle homomorphism defined by $b\mapsto b\dis(Q)_a$. 
Then $E$ is a connected quandle and $p$ is a central covering according to Proposition \ref{central iff}. In this case the stabilizer of the block of $1$ with respect to $\ker{p}$ is $\dis(Q)_a$. Hence if $\dis(Q)_a$ is not abelian then $Q$ is not an abelian cover of $Q/\alpha$ by Proposition \ref{character ab cov} (as already showed in \cite[Lemma B.6]{GenAlex}).  
\end{example}


%
%
The following theorem characterizes connected coverings of principal quandles.

\begin{theorem}\label{covering of principal}
Let $Q$ be a connected quandle and $\alpha\leq\lambda_Q$. The following conditions are equivalent:
\begin{itemize}
\item[(i)] $Q/\alpha$ is principal.
\item[(ii)] $Q$ is principal and $\dis^\alpha$ is transitive on each block of $\alpha$.
\end{itemize}
If the conditions hold then $Q$ is an abelian cover of $Q/\alpha$. 
%
\end{theorem}

\begin{proof}

(i) $\Rightarrow$ (ii) According to \cite[Proposition 2.9]{Principal} if $Q/\alpha$ is principal, then $\dis(Q)_{[a]_\alpha}=\dis^\alpha$ and it is transitive on each block of $\alpha$ according to Proposition \ref{prop:connected ext}. Then $\dis(Q)_a\leq \dis^\alpha\leq Z(\dis(Q))$ and so $\dis^\alpha$ is regular on each block of $\alpha$. Thus $\dis(Q)_a=1$ and so $Q$ is principal by Proposition \ref{caratt principal}. 

(ii) $\Rightarrow$ (i) Both the subgroups $\dis^\alpha$ and $\dis(Q)_{[a]}$ are regular on each block of $\alpha$ since $Q$ is connected and principal (see Proposition \ref{prop:connected ext} again) and their orbits correspond to their cosets. Therefore $\dis^\alpha=\dis(Q)_{[a]}$ and so $Q/\alpha$ is principal by virtue of \cite[Proposition 2.9]{Principal}.  

If (i) holds then $\dis^\alpha=\dis{(Q)}_{[a]}\leq Z(\lmlt{Q})$ is abelian. So we can apply Proposition \ref{character ab cov} and then $Q$ is an abelian cover of $Q/\alpha$.
\end{proof}

\begin{corollary}
Let $Q$ be an affine quandle. Then every connected quandle cover of $Q$ is principal and it is an abelian cover.
\end{corollary}

%
%

\begin{example}\label{Answer to V}
Let $A$ be an abelian group and $f\in \aut(A)$ such that $1-f$ is surjective but not injective. Then $ Q=\aff(A,f)$ is connected but not faithful and $1-f$ is a quandle homomorphism whose kernel is $\lambda_Q$. Hence we have an infinite chain of connected abelian covers
\begin{displaymath}
\ldots\longrightarrow Q\longrightarrow Q\longrightarrow  Q \longrightarrow \ldots
\end{displaymath}
in which at each step we take the factor $Q/\lambda_Q\simeq Q$.
This example answers in a positive way to \cite[Question 8.7]{CSV}. Note that $\dis(Q/\lambda_Q)\cong \dis(Q)$ but $\dis^{\lambda_Q}=\mathrm{Fix}(f)$, hence $\pi_{\lambda_Q}$ is not an isomorphism.

For instance, one can take any $n$-divisible abelian group $A$ (i.e. $nA=A$) with non-trivial $n$-torsion (i.e. $\setof{a\in A}{na=0}\neq 0$), and define $Q=\aff(A,n+1)$. 
Concrete examples are given by $\aff(\mathbb{Z}_{p^\infty},1+p)$ for every prime $p$ (here $\Z_{p^\infty}$ denotes the Pr\"ufer group), or $\aff(\mathbb{Q}/\Z,n)$ for every $n$. 
\end{example}

\section{The largest idempotent factor}

Let $Q$ be a rack. 
For a subset $X\subseteq Q$, we denote by $Sg(X)$ the subrack generated by $X$, the smallest subrack of $Q$ containing $X$. 
Note that $Sg(a)=\setof{L_a^k(a)}{k\in \mathbb{Z}}$ is a connected permutation rack, since $L_a^k(a)*u=L_{L_a^k(a)}(u)=L_a^kL_aL_a^{-k}(u)=a*u$, which is independent of $k$. 
Therefore, $Sg(a)$ is isomorphic either to $C_\infty=(\Z,*)$ with $x*y=y+1$, or to $C_n=(\Z_n,*)$ with $x*y=y+1\bmod n$.

Let $\m_Q$ denote the smallest congruence such that the corresponding factor is idempotent. Equivalently, the smallest congruence such that $a\,\m_Q\,(a*a)$ for every $a\in Q$ (see \cite[Section 2]{Sta-SI}).

\begin{proposition}
Let $Q$ be a rack. Then: 
\begin{enumerate}
	\item $[a]_{\m_Q}=Sg(a)$ for every $a\in Q$; 
	\item $\m_Q$ is strongly abelian and central;
	\item The mapping
	$$\pi_{\m_Q}:\aut(Q)\to\aut(Q/\m_Q),\quad \pi_{\m_Q}(h)([x])=[h(x)],$$ is a well defined group homomorphism;
	\item $Q$ is connected if and only if $Q/\m_Q$ is connected;
	\item if $Q$ is homogeneous, then $Q$ is an abelian cover of $Q/\m_Q$. 
\end{enumerate}
\end{proposition}

\begin{proof}
(1) 
We will show that $a\,\m_Q\,b$ if and only if $b\in Sg(a)$. The backward implication is clear. To show the forward implication, it is enough to check that $\alpha=\{(a,b):b\in Sg(a)\}$ is a congruence of $Q$, hence it must be the smallest congruence containing all pairs $(a,aa)$. Reflexivity is clear. For symmetry, if $b\in Sg(a)$, then $b=L_a^k(a)$, hence $a=L_a^{-k}(b)=L_b^{-k}(b)$, since $Sg(a)$ is a permutation rack. For transitivity, if $b=L_a^k(a)$ and $c=L_b^l(b)$, then $c=L_b^l(L_a^k(a))=L_a^{k+l}(a)$, for the same reason. Hence $\alpha$ is an equivalence. Let $b\in Sg(a)$, $b=L_a^k(a)$. Finally, $b*c=L_a^k(a)*c=a*c$, and thus $b*c\in Sg(a*c)$. Also $c*b=c*L_a^k(a)=L_{c*a}^k(c*a)$ using left distributivity, hence $c*b\in Sg(c*a)$.

(2) To prove that $\m_Q\leq\lambda_Q$, it is sufficient to observe that $L_{a*a}=L_{L_a(a)}=L_aL_aL_a^{-1}=L_a$, for every $a\in Q$.
To prove that $\m_Q$ is central, it remains to observe that $\dis(Q)$ is $\m_Q$-semiregular: if $h\in\dis(Q)$ fixes $a$, then $h$ fixes the whole $Sg(a)=[a]$, since $h$ is an automorphism.

(3) Let $h\in\aut(Q)$. Then $h(L_a^k(a))=L_{h(a)}^k(h(a))$ for every $a\in Q$. Therefore, if $a\,\m_Q\, b$, then $b=L_a^k(a)$ for some $k$, and we see that $h(a)\,\m_Q\, h(b)$. Hence the mapping $\pi_{\m_Q}$ is well defined.

(4) The forward implication is obvious. In the other direction, if $[a_1]\cdots[a_n][x]=[y]$ in $Q/\m_Q$, then $L_{a_1}\ldots L_{a_n}(x)\,\m_Q\, y$ in $Q$, and thus $L_y^kL_{a_1}\ldots L_{a_n}(x)=y$ for certain $k$.

(5) Observe that the mapping $s:Q\to Q$, $a\mapsto a*a$ is an automorphism of $Q$: indeed, $s(a*b)=(a*b)*(a*b)=a*(b*b)=s(a)*s(b)$, using the fact that blocks of $\m_Q$ are permutation racks in the last step. Let $A=\langle s\rangle\leq\aut(Q)$. Since $s^n(a)=L_a^n(a)$, we see that $[a]_{\m_Q}=a^A$. Since $Q$ is homogeneous, all blocks of $\m_Q$ are pairwise isomorphic subracks and $A_a=1$ for every $a$. Proposition \ref{abelian cov iff} finishes the proof.
\end{proof}

\begin{corollary}
Every rack is a central cover of a quandle. Every connected rack is an abelian cover of a connected quandle.
\end{corollary}

The following is an analog of Lemma \ref{lemma on cohomology}. The proof is similar and it is based on the fact that an isomorphism between two racks $E$ and $E^\prime$ induces an isomorphism between $E/\m_E$ and $E^\prime/\m_{E^\prime}$.

\begin{lemma}\label{lemma on cohomology2}
Let $Q$ be a rack, $\theta,\varepsilon$ rack cocycles over a set $A$, $E=Q\times_{\theta} A$, $E^\prime=Q\times_{\varepsilon} A$, $\pi$ be the canonical projection onto $Q$ and assume that $\m_{E}=\m_{E^\prime}=\ker{\pi}$.
Then $E\cong E^\prime$ if and only if there exists $g\in \aut(Q)$ such that $\theta$ and $\varepsilon\circ (g\times g)$ are cohomologous.
\end{lemma}

%
%

\bibliographystyle{abbrv}
\bibliography{references} 

\end{document}